\documentclass[12pt]{article}

\usepackage{amsthm}
\usepackage{amssymb}
\usepackage{amsmath}
\usepackage{comment}
\usepackage{thm-restate}
\usepackage{url} 
\usepackage{hyperref}
\usepackage[noabbrev,capitalise]{cleveref}

\usepackage{color}
\usepackage[normalem]{ulem}

\usepackage[margin=1in]{geometry}
\renewcommand{\v}{\textup{\textsf{v}}}
\newcommand{\e}{\textup{\textsf{e}}}
\renewcommand{\d}{\textup{\textsf{d}}}

\theoremstyle{plain}
\newtheorem{thm}{Theorem}[section]
\newtheorem{lem}[thm]{Lemma}
\newtheorem{claim}{Claim}[thm]
\newtheorem{proposition}[thm]{Proposition}

\newtheorem{conj}[thm]{Conjecture}

{\noindent \emph{Proof.} {}{#1}{}}{\hfill
	$\Diamond$\vspace{1em}}

\theoremstyle{plain} % just in case the style had changed
\newcommand{\thistheoremname}{}
\newtheorem{genericthm}[section]{\thistheoremname}

\theoremstyle{definition}
\newtheorem{definition}[thm]{Definition}

\title{Improved bound for Hadwiger's conjecture}

\author{ Yan Wang }

\begin{document}

\maketitle

\begin{abstract}
Hadwiger conjectured in 1943 that for every integer $t \ge 1$, every graph with no $K_t$ minor is $(t-1)$-colorable. 
Kostochka, and independently Thomason, proved every graph with no $K_t$ minor is $O(t(\log t)^{1/2})$-colorable.
Recently, Postle improved it to $O(t (\log \log t)^6)$-colorable. 
In this paper, we show that every graph with no $K_t$ minor is $O(t (\log \log t)^{5})$-colorable.
\end{abstract}

\section{Introduction}

Let $G$ be a graph.
A minor of $G$ is a graph obtained from $G$ by contracting edges, deleting edges and deleting isolated vertices. 
Minors play an important role in topological graph theory. 
In 1937, Wagner \cite{Wag37} showed that a graph is planar if and only if the complete graph on five vertices $K_5$ and the complete bipartite graph with three vertices in each partition $K_{3,3}$ are not minors of $G$.
In 1943, Hadwiger \cite{Had43} conjectured the following.

\begin{conj}[Hadwiger's conjecture~\cite{Had43}]\label{Hadwiger} For every integer $t \geq 1$, every graph with no $K_{t}$ minor is $(t-1)$-colorable. 
\end{conj}

Hadwiger's conjecture is a strengthening of the four-color theorem, is probably the most famous open problem in graph theory.
In fact, Hadwiger  \cite{Had43} proved the conjecture for $t \le 4$ and Wagner \cite{Wag37Equiv} established the equivalence of the case when $t=5$ and the four-color theorem.
Robertson, Seymour and Thomas \cite{RST93} showed Hadwiger's conjecture when $t=6$, but it is still open for $t \ge 7$.
For a complete survey and background of Hadwiger's conjecture, we refer the readers to \cite{Sey16Survey}.

Consider the following weakening of Hadwidger's conjecture: What can we show about the chromatic number of graphs with no $K_t$ minor?
Kostochka \cite{Kostochka82, Kostochka84} and Thomason \cite{Thomason84} showed that every graph with no $K_t$ minor 
is $O(t(\log t)^{1/2})$-degenerate, and thus is $O(t(\log t)^{1/2})$-colorable. 
Recently, Norin, Song and Postle \cite{NPS19} improved it to $O(t(\log t)^{\beta})$ for every $\beta > 1/4$.
Subsequently, it is further improved to $O(t(\log \log t)^{18})$ in \cite{Pos20} and $O(t(\log \log t)^{6})$ in \cite{Pos20LogLog}.
It is conjectured in \cite{Kaw07,KawMoh06,ReeSey98} that there exists a constant $C > 0$ such that for every integer $t \ge 1$, every graph with no $K_t$ minor is $Ct$-colorable. 
In this paper, we show the following.

\begin{thm}\label{t:ordinaryHadwiger3}
Every graph with no $K_t$ minor is $O(t (\log \log t)^{5})$-colorable.
\end{thm}

Our main contribution is the following improvement to the density increment lemma.

\begin{restatable}{lem}{Increment}\label{t:newforced} There exists a constant $C=C_{\ref{t:newforced}} > 0$ such that the following holds. Let $G$ be a graph with $\d(G) \ge C$, and let $D > 0$ be a constant. Let $s=D/\d(G)$ and let $g_{\ref{t:newforced}}(s) := C(1+ \log s)^{5}$.  Then $G$ contains at least one of the following: 
\begin{description}
		\item[(i)] a minor $J$ with $\d(J) \geq D$, or
		\item[(ii)] a subgraph $H$ with $\v(H) \leq g_{\ref{t:newforced}}(s) \cdot \frac{D^2}{\d(G)}$ and $\d(H) \geq \frac{\d(G)}{g_{\ref{t:newforced}}(s)}$.
	\end{description}  
\end{restatable}

We need the following theorem proved in~\cite{Pos20}.

\begin{lem}[Theorem 2.2 in~\cite{Pos20}]\label{thm:combined}
Every graph with no $K_t$ minor has chromatic number at most  
$$O\left(t \cdot\left(g_{\ref{t:newforced}}\left(3.2 \cdot \sqrt{\log t}\right) + (\log \log t)^2\right)\right).$$
\end{lem}
It is easy to see that Lemmas~\ref{t:newforced} and~\ref{thm:combined} imply Theorem~\ref{t:ordinaryHadwiger3}.

\subsection{Notations}

Let $G$ be a graph.
Let $V(G)$ and $E(G)$ be the vertex set and the edge set of $G$ respectively. 
We write the number of vertices $\v(G) = |V(G)|$, the number of edges $\e(G) = |E(G)|$ and density $\d(G) = \e(G) / \v(G)$.
For a vertex $v \in V(G)$, let ${\deg}_G(v)$ be the degree of $v$ in $G$ and $N_G(v)$ be the neighbourhood of $v$ in $G$.
We use $\delta(G)$ to denote the minimum degree of $G$.
For $S \subseteq E(G)$, we denote $G/S$ to be the graph obtained from $G$ by contracting all the edges in $S$.
For $A \subseteq V(G)$, we denote $G[A]$ to be the induced subgraph of $G$ on vertex set $A$.

\section{Preliminaries}\label{s:force}

In this section, we introduce some definitions and lemmas from \cite{Pos20LogLog}.

\begin{definition}
Let $G$ be a graph, and let $K,d\ge 1$, $\varepsilon \in (0,1)$ be real.
We say that 
\begin{itemize} 
\item a vertex of $G$ is \emph{$(K,d)$-small} in $G$ if ${\deg}_G(v) \le Kd$ and \emph{$(K,d)$-big} otherwise; 
\item two vertices of $G$ are \emph{$(\varepsilon,d)$-mates} if they have at least $\varepsilon d$ common neighbours; 
\item $G$ is \emph{$(K,\varepsilon_1, \varepsilon_2, d)$-unmated} if  every $(K,d)$-small vertex in $G$ have strictly fewer than $\varepsilon_1 d$  $(\varepsilon_2,d)$-mates. 
\end{itemize}
\end{definition}

If a graph is not unmated, it must contain a dense subgraph as shown by the following proposition (see Proposition 3.2 in \cite{Pos20LogLog}).

\begin{proposition}[\cite{Pos20LogLog}]\label{SmallDense}
	For all $K, d\ge 1$, $\varepsilon_1,\varepsilon_2 \in (0,1)$ and every graph $G$ at least one of the following holds:
	\begin{description}
		\item[(i)] there exists a subgraph $H$ of $G$ with $\v(H) \le 3Kd$ and $\e(H)\ge \varepsilon_1\cdot \varepsilon_2 \cdot \frac{d^2}{2}$, or
		\item[(ii)] $G$ is $(K,\varepsilon_1,\varepsilon_2,d)$-unmated.
	\end{description}
\end{proposition}

We need the following definitions about forests and bounded minors. 

\begin{definition}
Let $F$ be a non-empty forest in a graph $G$.
Let $K,k,d,s \ge 1$ be real and let $\varepsilon_2, c \in (0,1)$. We say $F$ is
	\begin{itemize}
		\item \emph{$(K,d)$-small} if every vertex in $V(F)$ is $(K,d)$-small in $G$, and
%		\item \emph{$(\varepsilon_2,d)$-mate-free} if no two distinct vertices in any component of $F$ are $(\varepsilon_2,d)$-mates in $G$, and
		\item \emph{$(c,d)$-clean} if $\e(G) - \e(G/F) \le c\cdot d \cdot \v(F)$,
		%\item \emph{$k$-bounded} if $\v(T)\le k$ for every component $T$ of $F$, and
		%\item a \emph{$(k,p)$-shrubbery} if $k-p < \v(T) \le k$ for every component $T$ of $F$.
	\end{itemize}
\end{definition}

\begin{definition}
Let $H$ and $G$ be two graphs with $V(H) = [h]$.
$(X_1,X_2,\dots,X_{h})$ is a \emph{model of $H$ in $G$} if
	\begin{itemize}
	\item $X_1,X_2,\dots,X_{h}$ are pairwise disjoint subsets of $V(G)$,
	\item $G[X_i]$ is connected for every $i \in [h]$, and
	\item there exists an edge between $X_i$ and $X_j$ in $G$ for every $ij \in E(H)$.
	\end{itemize}
Note that $G$ has an $H$ minor if and only if $G$ contains a model of $H$.
For integer $k \ge 1$, we say $G$ has a \emph{$k$-bounded} $H$ minor if $G$ contains a model of $H$ where $|X_i| \le k$ for $i \in [h]$.
\end{definition}

If a bipartite graph is almost complete on one partition, then it contains either a dense subgraph or a bounded minor with increased density (see Theorem 3.6 in \cite{Pos20LogLog}).

\begin{restatable}[\cite{Pos20LogLog}]{lem}{kClawDense}\label{kclawDense}
Let $K_0, \ell_0 \geq 2$ be  integers with $K_0 \ge \ell_0(\ell_0+1)$, and let $\varepsilon_{1,0}\in \left(0,\frac{1}{\ell_0}\right], \varepsilon_{2,0}\in \left(0,\frac{1}{\ell_0^2}\right]$ and  $d_0 \ge 1/\varepsilon_{2,0}$ be real. Let $G=(A,B)$ be a bipartite graph such that $|A| \ge \ell_0 |B|$ and every vertex in $A$ has at least $d_0$ neighbours in $B$. Then there exists at least one of the following:
	\begin{description}
		\item[(i)] a subgraph $H$ of $G$ with $\v(H) \le 4 K_0 d_0$ and $\e(H)\ge \varepsilon_{1,0}\cdot \varepsilon_{2,0}\cdot d_0^2/2$.
		\item[(ii)] a subgraph $H$ of $G$ with $\v(H) \le 4\ell_0 K_0 d_0$ and $\e(H)\ge \varepsilon_{1,0}^2 \cdot d_0^2/2$.
		\item[(iii)] an $(\ell_0+1)$-bounded minor $H$ of $G$ with $\d(H) \ge \frac{\ell_0^2}{\ell_0+1}\left(1-2\varepsilon_{1,0} - 2\ell_0\varepsilon_{2,0} -\frac{\ell_0}{K_0}\right)d_0$.
	\end{description}
\end{restatable}

\section{Dense Subgraphs or Minors with Increased Density}\label{Shrub}

In this section we prove the following lemma which shows every graph must contain a dense subgraph, a bipartite subgraph such that all the vertices on one partition have large degree, or a bounded minor with increased density. 

\begin{restatable}{lem}{DenseThree}\label{DenseSubgraph3}
Let $K \ge k \ge 100$ be integers with $K\ge 4 k^{2}$. Let $\ell = \left\lceil \frac{k}{6} \right\rceil$, $\varepsilon_1 \in (0,\frac{1}{k}]$ and $\varepsilon_2 \in (0,\frac{1}{k}]$. Let $G$ be a graph with $d=\d(G) \ge \frac{k}{\min\{\varepsilon_1, \varepsilon_2\}}$. Then $G$ contains at least one of the following:
\begin{description}
	\item[(i)] a subgraph $H$ of $G$ with $\v(H) \le 3Kd$ and $\d(H)\ge \frac{\varepsilon_1 \varepsilon_2  d}{6Kk}$, or
	\item[(ii)] a bipartite subgraph $H=(X,Y)$ of $G$ with $|X| \ge \ell |Y|$ such that  every vertex in $X$ has at least $(1-6\varepsilon_1)d$ neighbours in $Y$, or
	\item[(iii)] a $k$-bounded minor $G'$ of $G$ with $\d(G') \ge k \left(1 - \frac{30}{k}\right) d.$
\end{description}
\end{restatable}

\begin{proof}
For, otherwise, let $G$ be a minimal counterexample, i.e. $G$ satisfies none of (i), (ii) or (iii), and $\v(G)$ is minimized.
For any proper subgraph $H$ of $G$, since $H$ is not a minimal counterexample, we may assume $\d(H) < \d(G)$.
In particular, for any vertex $v \in V(G)$, $\d(G[V(G) \setminus \{v\}]) < \d(G)$.
This implies $d(v) > \d(G)$, so $\delta(G) > \d(G)$.

First we apply Proposition \ref{SmallDense} to $G$ with $(K, \varepsilon_1, \varepsilon_2, d)_{\ref{SmallDense}} = (K, \varepsilon_1, \frac{\varepsilon_2}{k},  d )$.
If Proposition \ref{SmallDense} (i) holds, then there exists a subgraph $H$ of $G$ such that $\v(H) \le 3K d)$, $\e(H) \ge \frac{\varepsilon_1 \varepsilon_2 d^2}{2k}$. 
So $\d(H) \ge \frac{\varepsilon_1 \varepsilon_2 d}{6Kk}$ and conclusion (i) holds, a contradiction. 
Hence, Proposition \ref{SmallDense} (ii) holds, i.e. $G$ is $(K, \varepsilon_1, \frac{\varepsilon_2}{k}, d)$-unmated.  

Let $A = \{v: \deg_G(v) \le K d \}$ and $B = V(G) \backslash A$.
Since $K d |B| \le 2\e(G) = 2 d \v(G)$, we have $|B| \le \frac{2}{K}\v(G) \le \frac{1}{2k^2}\v(G)$.
%Moreover, $$|A| \ge \v(G) - |B| \ge (1-\frac{1}{2k^2})v(G) \ge 0.99 \v(G) \ge 2l \cdot \frac{1}{2k^2}\v(G) \ge 2l |B|.$$
Since $G$ is $(K, \varepsilon_1, \frac{\varepsilon_2}{k}, d)$-unmated, $v$ has fewer than $\varepsilon_1 d$ $(\frac{\varepsilon_2}{k}, d)$-mates in $G$ for every $v \in A$.

Let $c = 4\varepsilon_2$ and $F$ be a maximal $(K,d)$-small, $(c,d)$-clean forest where each component of $F$ is a star of size $k$.  
Since $F$ is $(K,d)$-small, we have $V(F) \subseteq A$.

\medskip

\begin{claim}\label{newclaim}
If $F_0$ is a star forest in $G$ and $v \in V(G) \backslash V(F_0)$ has at least $2\varepsilon_1 d$ neighbours in $A \backslash V(F_0)$,
then there exists a star $T$ of size $k$ in $A \backslash V(F_0)$ with center $v$ such that 
$$\e(G/E(F_0)) - \e(G/(E(F_0) \cup E(T)) ) \le 2k\varepsilon_2 d.$$
\end{claim}

\begin{proof}
For, otherwise, let $S$ be a maximal star in $A \backslash V(F_0)$ with center $v$ such that
$$\e(G/E(F_0)) - \e(G/(E(F_0) \cup E(S)) ) \le 2 (\v(S) - 1) \varepsilon_2 d.$$
Such $S$ exists as $S$ could be $\{v\}$.
By assumption, $\v(S) \le k - 1$.

Let $V(S) = \{v_0,v_1,\dots,v_s\}$ with $v_0 = v$ and $s \le k-2$.
Let $U = \{u_1,\dots,u_n\}$ where for each $j \in [n]$, $u_j \in A \backslash (V(F_0) \cup V(S))$ and there exists $i \in \{0,1,\dots, s\}$ such that $u_j$ is an $(\frac{\varepsilon_2}{k}, d)$-mate of $v_i$.
Since $G$ is $(K, \varepsilon_1, \frac{\varepsilon_2}{k}, d)$-unmated, $v_i$ has fewer than $\varepsilon_1 d$ $(\frac{\varepsilon_2}{k}, d)$-mates in $G$ for $i \in \{0,1,\dots, s\}$.
So $n < k \varepsilon_1 d$.

Let $G' = G/(E(F_0) \cup E(S))$ and $v_S$ be the vertex in $V(G')$ corresponding to $S$.
For $u \in A \backslash (V(F_0) \cup V(S) \cup U)$, by definition of $U$, $u$ has fewer than $\frac{\varepsilon_2}{k}d$ common neighbours with $v_i$ in $G$ for all $i \in \{0,1,\dots, s\}$.
So $u$ has fewer than $(s+1) \frac{\varepsilon_2}{k}d < \varepsilon_2 d$ common neighbours with $v_S$ in $G'$.

Now we show that at most $\varepsilon_1 d$ vertices in $U$ have at least $\varepsilon_2 d$ common neighbours with $v_S$ in $G'$.
Consider the following auxiliary graph.
Let $W$ be an edge-weighted complete bipartite graph with vertex partition $V(S) \cup U$.
For $i \in \{0,1,\dots, s\}$ and $j \in [n]$, we define the edge weight $w_{ij}$ to be the number of common neighbours between $v_i$ and $u_j$.

We claim that for $i \in \{0,1,\dots, s\}$, $\sum_{j \in [n]} w_{ij} < \frac{\varepsilon_1 \varepsilon_2 d^2}{k}$.
For, otherwise, suppose there exists some $i \in \{0,1,\dots, s\}$ such that $\sum_{j \in [n]} w_{ij} \ge \frac{\varepsilon_1 \varepsilon_2 d^2}{k}$.
Let $H = G[\{v_i\} \cup N(v_i) \cup U]$.
Then $\v(H) \le 1 + Kd + n < 2Kd$.
By definition of $W$, $\e(H) \ge \sum_{j \in [n]} w_{ij} \ge \frac{\varepsilon_1 \varepsilon_2 d^2}{k}$.
Hence, $\d(H) \ge \frac{\varepsilon_1 \varepsilon_2 d}{2Kk}$ and conclusion (i) holds, a contradiction.

Let $\Gamma = \{j \in [n]: \sum_{i \in \{0,1,\dots, s\}} w_{ij} \ge \varepsilon_2 d\}$.
We have 
$$|\Gamma| \varepsilon_2 d \le \sum_{i \in \{0,1,\dots, s\}} \sum_{j \in [n]} w_{ij} \le (s+1) \cdot \frac{\varepsilon_1 \varepsilon_2 d^2}{k} < \varepsilon_1 \varepsilon_2 d^2.$$
So $|\Gamma| < \varepsilon_1 d$.
This implies that fewer than $\varepsilon_1 d$ vertices in $U$ have at least $\varepsilon_2 d$ common neighbours with $v_S$ in $G'$.

Since $v$ has at least $2 \varepsilon_1 d$ neighbours in $A \backslash V(F_0)$, there exists a vertex $u \in A \backslash (V(F_0) \cup V(S))$ that is not a $(\varepsilon_2, d)$-mate of $v_S$ in $G'$.

Let $S'$ be the star with $V(S') = V(S) \cup \{u\}$ and $E(S') = E(S) \cup \{uv\}$.
Note that $\v(S') > \v(S)$ and 
\begin{equation*}
\begin{split}
\quad \quad & \e(G/E(F_0)) - \e(G/(E(F_0) \cup E(S')) ) \\
= \quad& (\e(G/E(F_0)) - \e(G/(E(F_0) \cup E(S)) )) + (\e(G/(E(F_0) \cup E(S)) ) - \e(G/(E(F_0) \cup E(S')) )) \\
\le \quad& (2 (\v(S) - 1) \varepsilon_2 d) + (1 +  \varepsilon_2 d) \\
\le \quad& 2 (\v(S') - 1) \varepsilon_2 d.
\end{split}
\end{equation*}
This is a contradiction to the maximality of $S$.
\end{proof}

\medskip

Next we work on the graph obtained from $G$ by contracting $F$.
Let $G_0=G/E(F)$. 
Since each component of $F$ is a star of size $k$, $G_0$ is a $k$-bounded minor of $G$. 
Let $A' = A \setminus V(F)$.
By Claim \ref{newclaim} and the maximality of $F$, we can show that the edges between $F$ and $A'$ are sparse in the following claim.

\medskip

\begin{claim}\label{onlyone}
Every component $T$ of $F$ has at most one vertex with at least $3\varepsilon_1 d$ neighbours in $A'$ in graph $G$.
\end{claim}

\begin{proof}
For, otherwise, suppose there exist distinct $u_1, u_2\in V(T)$ such that $|N_G(u_i)\cap A'|\ge 3\varepsilon_1 d$. Let $F_0 = F\setminus V(T)$. Note that since $E(F_0)\subseteq E(F)$, we have that $\e(G/E(F_0))\ge \e(G/E(F)).$

Since $F$ is $(c,d)$-clean, we have
$$\e(G)-\e(G/E(F_0)) \le \e(G)-\e(G/E(F)) \le c d  \v(F).$$
Since $u_1$ has at least $3\varepsilon_1 d$ neighbours in $A'$, by Claim~\ref{newclaim} there exists a star $T_1$ of size $k$ in $A\setminus V(F_0)$ with center $u_1$ such that 
$$\e(G/E(F_0)) - \e(G/(E(F_0)\cup E(T_1))) \le 2k \varepsilon_2 d.$$ 

Let $F_1$ be the union of $F_0$ and $T_1$, 
i.e. $V(F_1) = V(F_0) \cup V(T_1)$ and $E(F_1) = E(F_0) \cup E(T_1)$. 
Since $u_2$ has at least $3 \varepsilon_1 d$ neighbours in $A'$,  $u_2$ has at least $3 \varepsilon_1 d - k \ge 2 \varepsilon_1 d$ neighbours in $A'\setminus V(T_1)$.
Again by Claim~\ref{newclaim}, there exists a star $T_2$ of size $k$ in $A \setminus V(F_1)$ with center $u_2$ such that 
$$\e(G/E(F_1)) - \e(G/(E(F_1)\cup E(T_2))) \le 2k \varepsilon_2 d.$$

Let $F_2$ be the union of $F_1$ and $T_2$, 
i.e. $V(F_2) = V(F_1) \cup V(T_2)$ and $E(F_2) = E(F_1) \cup E(T_2)$. 
Note that $\v(F_2) = \v(F_0) + \v(T_1) + \v(T_2) = \v(F_0) + 2k = \v(F) + k$. 
Moreover,
\begin{equation*}
\begin{split}
\quad \quad & \e(G) - \e(G/E(F_2)) \\
=\quad & (\e(G) - \e(G/E(F_0))) + (\e(G/E(F_0)) - \e(G/E(F_1))) + (\e(G/E(F_1)) - \e(G/E(F_2))) \\
\le \quad & cd\v(F) + 2k\varepsilon_2d + 2k\varepsilon_2d \\
= \quad & 4 (\v(F) + k)\varepsilon_2d \\
= \quad & cd\v(F_2).
\end{split}
\end{equation*}
So $F_2$ is $(c,d)$-clean. 
Note that $F_2$ is a forest where every component is a star of size $k$ and $F_2$ is also $(K,d)$-small.
This contradicts the maximality of $F$.
\end{proof}

\medskip

By similar argument, we can show that edges within $A'$ are also sparse in the following claim.
 
\medskip 
 
\begin{claim}\label{newAPrime}
Every vertex $v$ in $A'$ has at most $2 \varepsilon_1 d$ neighbours in $A'$ in $G$.
\end{claim}

\begin{proof}
For, otherwise, suppose $v$ has at least $2\varepsilon_1 d$ neighbours in $A'$ in $G$. 
By Claim~\ref{newclaim}, there exists a star $T$ of size $k$ in $A\setminus V(F_0)$ with center $v$ such that 
$$\e(G/E(F)) - \e(G/(E(F)\cup E(T))) \le 2k \varepsilon_2 d \le c d k.$$ 
Let $F'$ be the union of $F$ and $T$,
i.e. $V(F') = V(F) \cup V(T)$ and $E(F') = E(F) \cup E(T)$. 
Then $F'$ is a $(K,d)$-small forest where every component is a star of size $k$. 
Moreover, $\v(F')=\v(F)+k$. 
Since $F$ is $(c,d)$-clean, we have
$$\e(G)-\e(G/E(F)) \le c d \v(F).$$
Hence, 
$$\e(G)-\e(G/E(F')) \le c d \v(F) + cdk = cd\v(F').$$
So $F'$ is also $(c,d)$-clean.
This contradicts the maximality of $F$.
\end{proof}

\medskip

Let $C$ be the set of vertices in $F$ with at least $3 \varepsilon_1 d$ neighbours in $A'$. 
By Claim~\ref{onlyone}, every component of $F$ has at most one vertex in $C$. 
Hence $|C|\le \frac{1}{k} \v(G)$.
Now we show that $A'$ is small.

\medskip

\begin{claim}\label{APrimeSize}
$|A'|\le \frac{2\v(G)}{3}$.
\end{claim}			
\begin{proof}
For, otherwise, suppose $|A'| > \frac{\v(G)}{2}$.
First we aim to construct a bipartite graph that satisfies conclusion (ii).
Let $A_1 = \{v\in A', |N_G(v)\cap (B\cup C)| \ge (1 - 6 \varepsilon_1) \d(G) \}$ and let $A_2 = A'\setminus A_1$. 
Suppose that $|A_1| \ge \left(\frac{1}{3} + \frac{2}{k}\right)  \v(G)$. 
Since $|B\cup C|\le \frac{2}{k} \v(G)$, we have that $|A_1| \ge \left(\frac{k}{6}+1\right) |B \cup C| \ge \ell |B \cup C|$ and hence bipartite graph $(A_1, B \cup C)$ satisfies conclusion (ii), a contradiction.  

So we may assume that $|A_1| < \left(\frac{1}{3} + \frac{2}{k}\right) \v(G)$. 
By Claim~\ref{newAPrime}, every vertex in $A'$ has at most $2\varepsilon_1 d$ neighbors in $A'$. 
Since $\delta(G)\ge \d(G) = d$, it follows that every vertex in $A_2$ has at least $4\varepsilon_1 d$ neighbors in $V(F)\setminus C$. 
Hence
$$ \e(G(A_2,V(F)\setminus C)) \ge 4\varepsilon_1 d \cdot |A_2|.$$
By definition of $C$, we have that
$$ \e(G(A_2,V(F)\setminus C)) \le 3\varepsilon_1 d \cdot |V(F) \setminus C| \le 3\varepsilon_1 d (\v(G) - |A_1| - |A_2|).$$
Hence $|A_2|\le  \frac{3}{7} (\v(G) - |A_1|)$. 
Therefore, we have 
$$|A'| = |A_1|+|A_2| \le \frac{3}{7} \v(G) + \frac{4}{7} |A_1| \le \frac{3}{7} \v(G) + \frac{4}{7} \left(\frac{1}{3} + \frac{2}{k}\right) \v(G) \le \frac{2\v(G)}{3}.$$
\end{proof}

\medskip

Finally let $G' = G_0 \backslash A'$. 
In the rest of the proof, we show that $G'$ satisfies conclusion (iii). 
Note that by Claim~\ref{APrimeSize}, we have that $|A'\cup B\cup C| \le (\frac{2}{3} + \frac{2}{k}) \v(G) < \v(G)$.
Hence $F$ is nonempty. 
In addition, $G[A'\cup B\cup C]$ is a proper subgraph of $G$. 
So $\d(G[A'\cup B\cup C]) < d$ and
$$\e(G[A'\cup B\cup C]) < d \cdot |A'\cup B \cup C| \le \left(\frac{2}{k}\v(G) + |A'|\right) d.$$
Moreover, by definition of $C$ we have
$$\e(G(A', V(F)\setminus C)) \le 3\varepsilon_1 d \cdot |V(F)\setminus C| < 3\varepsilon_1 d \cdot \v(G).$$
Let $a' = \frac{|A'|}{\v(G)}$. Hence
$$\e(G)-\e(G\setminus A') \le \left(\frac{2}{k} +3\varepsilon_1 + a'\right) d \cdot \v(G).$$
Since $F$ is $(c,d)$-clean, we have that 
$$\e(G) - \e(G_0) \le c d \cdot \v(F) \le 4\varepsilon_2 d \cdot \v(G).$$
Hence
$$\e(G\setminus A') - \e(G') \le 4\varepsilon_2 d \cdot \v(G),$$
and so we can lower bound $\e(G)'$ by
$$\e(G) - \e(G') \le  \left(\frac{2}{k} +3\varepsilon_1 + 4\varepsilon_2 + a'\right) d \cdot \v(G).$$
Moreover, since $G'$ is $k$-bounded minor, we can upper bound $\v(G')$ by
$$\v(G') = |B| + |C| \le \frac{1}{2k^2}\cdot \v(G) + \frac{\v(G)-|A'|}{k} < \frac{\v(G)}{k} \left( \frac{k+1/2}{k} - a'\right).$$
Thus
$$\d(G') \ge \frac{\e(G) - \left(\frac{2}{k} +3\varepsilon_1 + 4\varepsilon_2+a'\right)\cdot d \cdot \v(G)}{ \frac{\v(G)}{k} \left( \frac{k+1/2}{k} - a'\right) }
= k  d \cdot \frac{ 1 - \left(\frac{2}{k} +3\varepsilon_1 + 4\varepsilon_2\right) - a'}{\frac{k+1/2}{k} - a'}.$$
 
Since $a' \le \frac{2}{3}$, we have
\begin{align*}
\d(G') &\ge k d \cdot \frac{ 1 - \left(\frac{2}{k} +3\varepsilon_1 + 4\varepsilon_2\right) - \frac{2}{3}}{\frac{k+1/2}{k} - \frac{2}{3}}\\
&= k  d \cdot \frac{ 1 - 3\left(\frac{2}{k} +3\varepsilon_1 + 4\varepsilon_2\right) }{1 + \frac{3/2}{k}}\\
&= k  d \cdot \left(1 - 3\left(\frac{2}{k} +3\varepsilon_1 + 4\varepsilon_2\right)\right) \cdot \left(1 - \frac{1}{k}\right) \\
&\le k  d \cdot \left(1 - \frac{30}{k}\right) ,
\end{align*}
since $\varepsilon_1 \le \frac{1}{k}$ and $\varepsilon_2 \le \frac{1}{k}$. But now conclusion (iii) holds, a contradiction.
\end{proof}

\medskip

\section{Proof of Lemma \ref{t:newforced}}\label{s:outline}

In this section, we prove Lemma~\ref{t:newforced}.
First, we show the following lemma that is analogous to Theorem 2.1 in \cite{Pos20LogLog}.
The proof is similar, but we include it for completeness.

\medskip

\begin{restatable}{lem}{Dense}\label{DenseSubgraph2}
Let $k\ge 100$ be an integer. Let $G$ be a graph with $d=\d(G) \ge k^2$. Then $G$ contains at least one of the following:
	\begin{description}
		\item[(i)] a subgraph $H$ with $\v(H)\le 12 k^{3} d$ and $\d(H)\ge \frac{d}{24 k^5}$, or
	  \item[(ii)] an $m$-bounded minor $G'$ with $\d(G') \ge m \cdot \left(1-\frac{30}{m}\right) \cdot d$ for some integer $m\in [\frac{k}{6},k]$.
		\end{description}  
\end{restatable}

\begin{proof}[Proof of Theorem~\ref{DenseSubgraph2}]
	
We apply Lemma~\ref{DenseSubgraph3} to $G$ with $K=4 k^{2}$ and $\varepsilon_1=\frac{1}{k}$ and $\varepsilon_2=\frac{1}{k}$. 

If Lemma~\ref{DenseSubgraph3}(i) holds, $G$ contains a subgraph $H$ with $\v(H)\le 3Kd = 12 k^2 d$ and $\d(H)\ge \frac{\varepsilon_1 \varepsilon_2  d}{6Kk} = \frac{d}{24k^5}$. So (i) holds. 
If Lemma~\ref{DenseSubgraph3}(iii) holds, $G$ contains a $k$-bounded minor $G'$ with $\d(G') \ge k \left(1 - \frac{30}{k}\right)  d.$ So (ii) holds with $m=k$.

We may assume that Lemma~\ref{DenseSubgraph3}(ii) holds, 
i.e., there exists a bipartite subgraph $H=(X,Y)$ with $|X| \ge \frac{k}{6} |Y|$ such that every vertex in $X$ has at least $(1-6\varepsilon_1)d$ neighbours in $Y$. 
Now we apply Lemma~\ref{kclawDense} to $H$ with $(d_0,l_0,K_0, \varepsilon_{1,0},\varepsilon_{2,0})_{\ref{kclawDense}}= ((1-6\varepsilon_1)d,\frac{k}{6},\frac{k}{6}(\frac{k}{6}+1),\frac{1}{k/6},\frac{1}{(k/6)^2})$. 
Note that $d_0\ge d/2$ since $k\ge 12$ and hence $d_0\ge k^2/2 \ge \frac{1}{\varepsilon_{2,0}}$ as needed.

If Lemma~\ref{kclawDense}(i) holds for $H$, $H$ contains a subgraph $H_0$ with $\v(H_0) \le 4 K_0 d_0 \le 4k^2d$ and $\e(H_0)\ge \varepsilon_{1,0} \varepsilon_{2,0}\frac{d_0^2}{2}$. 
Then 
$\d(H_0) \ge \frac{\varepsilon_{1,0} \varepsilon_{2,0} d_0}{8K_0} \ge \frac{d}{16 l_0^5} \ge \frac{d}{24 k^5},$
so (i) holds.

If Lemma~\ref{kclawDense}(ii) holds for $H$, $H$ contains  a subgraph $H_0$ of $H$ with $\v(H_0) \le 4 \ell_0 K_0 d_0 \le 4k^3d$ and $\e(H_0)\ge \varepsilon_{1,0}^2 \frac{d_0^2}{2}$. 
Then
$\d(H_0)\ge \varepsilon_{1,0}^2 \frac{d_0}{8l_0 K_0} \ge \frac{d}{16l_0^5} \ge \frac{d}{24 k^5},$
so (i) holds.

Finally we may assume Lemma~\ref{kclawDense}(iii) holds, 
i.e., $H$ contains an $(\ell_0 +1)$-bounded minor $H_0$ with 
\begin{align*}
d(H_0) &\ge \frac{\ell_0^2}{\ell_0+1}\left(1-2\varepsilon_{1,0} - 2\ell_0\varepsilon_{2,0} -\frac{\ell_0}{K_0}\right)d_0\\
&\ge  (\ell_0 + 1) \cdot \left(1 - \frac{1}{l_0+1} \right)^2 \left(1-\frac{5}{\ell_0}\right)\cdot \left(1-6\varepsilon_1\right) \cdot d\\
&\ge (\ell_0+1) \cdot \left(1 -\frac{30}{\ell_0+1}\right)\cdot d,\\
\end{align*}
since $\ell_0 = \frac{k}{6}$. 
Therefore (ii) holds with $G'=H_0$ and $m=\ell_0+1$.
\end{proof}

\medskip

Now we are ready to derive \cref{t:newforced} from \cref{DenseSubgraph2}.

\begin{proof}[Proof of \cref{t:newforced}]
Let $C_{\ref{t:newforced}} = 2^{50}$. We proceed by induction on $s$. If $s\le 1$, then $J=G$ is a minor of $G$ with $\d(J)=\d(G) \ge s \d(G) = D$ and (i) holds as desired. So we may assume that $s > 1$. Hence $D > \d(G) \ge C_{\ref{t:newforced}}$. 

If $g_{\ref{t:newforced}}(s) \ge 2 \d(G)$, then let $H$ be the graph of a single edge $uv$ where $uv\in E(G)$ and (ii) holds since 
$\v(H) = 2 < C_{\ref{t:newforced}} < D^2$
and $\d(H) = 1/2$. 
So we may assume that $\d(G) > g_{\ref{t:newforced}}(s)/2$.

Let $k = \frac{1}{2} \cdot (C_{\ref{t:newforced}})^{\frac{1}{5}} \cdot (1+\log s) = 2^{9} \cdot (1+\log s)$. 
Note that $g_{\ref{t:newforced}}(s) = (2k)^5$.
Since $\log s\ge 0$, we have that 
$k \ge 512.$ 
Moreover, $\d(G) \ge g_{\ref{t:newforced}}(s)/2 \ge (2k)^5/ 2 \ge k^2$.

Now we apply \cref{DenseSubgraph2} to $G$ and $k$. 
Note that $\d(G) \ge k^{2}$ as needed.
If \cref{DenseSubgraph2}(i) holds, $G$ contains a subgraph $H$ with $\v(H)\le 12 k^{3} \d(G)$ and $\d(H)\ge \frac{\d(G)}{24k^5}$. 
Note that
$$\v(H) \le 12 k^{3} \d(G) \le g_{\ref{t:newforced}}(s) \d(G) \le g_{\ref{t:newforced}}(s) \frac{D^2}{\d(G)}$$
and furthermore
$$\d(H) \ge \frac{\d(G)}{24 k^{5}} \ge \frac{\d(G)}{g_{\ref{t:newforced}}(s)}.$$
Then (ii) holds as desired.

So we may assume that \cref{DenseSubgraph2}(ii) holds,
i.e., $G$ contains an $m$-bounded minor $G'$ with 
$\d(G') \ge m \left(1-\frac{30}{m}\right) \d(G)$
for some integer $m\in [\frac{k}{6}, k]$. 
Let $s' = D/\d(G')$. Note that since $k \ge 512$, we have that $m \ge \frac{k}{6} \ge 60$. Hence
$$\d(G') \ge m  \left(1-\frac{30}{m}\right) \d(G) \ge \frac{m}{2}  \d(G) > \d(G),$$
and 
$$s' \le \frac{s}{m \left(1-\frac{30}{m}\right)} \le \frac{2s}{m} <  s.$$

Since $s' < s$, we have by induction that at least one of (i) or (ii) holds for $G'$. 
If (i) holds for $G'$,
i.e., $G'$ contains a minor $J$ with $\d(J)\ge D$. 
Then $J$ is also a minor of $G$ and (i) holds for $G$.

So we may assume that (ii) holds for $G'$,
i.e., $G'$ contains a subgraph $H'$ with 
$\v(H') \le g_{\ref{t:newforced}}(s') \frac{D^2}{\d(G')}$
and 
$\d(H') \ge \frac{\d(G')}{g_{\ref{t:newforced}}(s')}.$ 
Note that $H'$ corresponds to a subgraph $H$ of $G$ with $\v(H) \le m \v(H')$ and $\e(H)\ge \e(H')$. Then
\begin{align*}
\v(H) &\le m \v(H') \le m g_{\ref{t:newforced}}(s') \frac{D^2}{\d(G')} \le \left(\frac{g_{\ref{t:newforced}}(s')}{1-\frac{30}{m}}\right) \frac{D^2}{d(G)}.
\end{align*}
and
$$\d(H) = \frac{\e(H)}{\v(H)} \ge \frac{\e(H')}{m \v(H')} = \frac{\d(H')}{m} \ge \frac{\d(G')}{m g_{\ref{t:newforced}}(s')} \ge \left(\frac{1-\frac{30}{m}} {g_{\ref{t:newforced}}(s')} \right) d(G).$$

In the rest of the proof, we show that $H$ satisfies (ii).
Since 
$m \ge \frac{k}{6} \ge 30(1+\log s)$,
we have
$$\frac{1}{1-\frac{30}{m}} \le 1 + \frac{60}{m} \le 1 + \frac{2}{1+\log s}.$$
On the other hand, since $m \ge \frac{k}{6} > e^3$, we have
\begin{align*}
\log s' &\le \log \left(\frac{2s}{m} \right) \le \log(s) + 1 - \log(m) \le \log(s)-2.
\end{align*}
Thus
$$\frac{ g_{\ref{t:newforced}}(s') } {g_{\ref{t:newforced}}(s)} \le \frac{(1+\log s')^{5}}{(1+\log s)^{5}} \le \frac{ 1+\log s'}{1+\log s} \le \frac{ 1+\log(s) - 2}{1+\log s} = 1 - \frac{2}{1+\log s}.$$
Now we have
$$\frac{g_{\ref{t:newforced}}(s')}{1-\frac{30}{m}} \le \left(1 - \frac{2}{1+\log s}\right)\left(1+\frac{2}{1+\log s}\right) g_{\ref{t:newforced}}(s) \le g_{\ref{t:newforced}}(s).$$
Hence,
$\v(H) \le  g_{\ref{t:newforced}}(s) \frac{D^2}{\d(G)}$ and
$\d(H) \ge \frac{\d(G)}{g_{\ref{t:newforced}}(s)}.$
Therefore, (ii) holds.
\end{proof}

\bibliographystyle{abbrv}
\bibliography{hadwiger_improved}

\end{document}